\documentclass[12pt, oneside]{amsart}       % use "amsart" instead of "article" for AMSLaTeX format

\usepackage[T1]{fontenc}
\usepackage{palatino}

\usepackage{graphicx}

\usepackage{amsaddr}

                            \usepackage{amsthm}    
\newtheorem{theorem}{Theorem}[section]

\newtheorem{definition}{Definition}[section]
\newtheorem{lemma}[theorem]{Lemma}

\theoremstyle{remark}
\newtheorem*{remark}{Remark}
\usepackage{tcolorbox}
\newtheorem*{example}{Example}

\title{Logarithmic jets and the chiral de Rham complex of a pair}
\author{Emile Bouaziz}
\address{Academia Sinica, Taipei}
\email{emile.g.bouaziz@gmail.com}

\begin{document} \maketitle\begin{abstract} To a smooth variety $X$ with simple normal crossings divisor $D$, we associate a sheaf of vertex algebras on $X$, denoted $\Omega^{ch}_{X}(\operatorname{log}D)$, whose conformal weight $0$ subspace is the algebra $\Omega_{X}(\operatorname{log}D)$ of forms with log poles along $D$.
We prove various basic structural results about $\Omega^{ch}_{X}(\operatorname{log}D)$. In particular, if $X^{*}=X\setminus D$ has a volume form then we show that $\Omega^{ch}_{X}(\operatorname{log}D)$ admits a topological structure of rank $d=\operatorname{dim}(X)$, which is enhanced to an extended topological structure if $D\sim -K_{X}$ is in fact anticanonical.
In this latter case we also show that the resulting $(q,y)$ character $\operatorname{Ell}(X,D)(q,y)$ is a section of the line bundle $\Theta^{\otimes d}$ on the elliptic curve $E=\mathbf{C}^{*}/q^{\mathbf{Z}}$. 
We further show how $\Omega^{ch}_{X}(\operatorname{log}D)$ can be understood in terms of a simple birational modification of the space of jets into $X$.\end{abstract}

\section{Introduction and statement of results} Our interest is in a sheaf of vertex algebras on a smooth variety $X$, depending on the data of a simple normal crossings divisor $D$ in $X$. In this introduction we sketch some motivation for our construction, and recall in particular the well known story when the divisor is taken to be empty. For a discussion of the precise geometric context we work in see \ref{context} below.

\subsection{The classical story} In a celebrated paper of Malikov, Schechtman and Vaintrob, \cite{MSV}, the authors associate to a smooth $\mathbf{C}$-variety $X$ a sheaf of vertex algebras on $X$, formally locally equivalent to a tensor product of $d:=\operatorname{dim}(X)$ copies of the $bc\beta\gamma$-system. The resulting sheaf is denoted $\Omega^{ch}_{X}$ and referred to as the \emph{chiral de Rham complex}, or the \emph{curved} $bc\beta\gamma$ system on $X$. It is by now well studied. A physical interpretation of this object in terms of sigma-models is given by Kapustin in \cite{Kap}. Further, a construction in the Beilinson-Drinfeld language of factorisation algebras (\emph{cf.} \cite{Bei}) is given by Kapranov and Vasserot in \cite{KapVass}, although we remark that this treatment does not obviously yield the topological structure. The authors of \cite{MSV} further show that if $X$ is Calabi-Yau then $\Omega^{ch}_{X}$ admits a large super-algebra of symmetries. Namely, there are well defined global sections of $\Omega^{ch}_{X}$ generating a copy of the \emph{topological vertex algebra} at rank $d$, a certain \emph{topological twist} of the $\mathcal{N}=2$ superconformal algebra.

\label{geom}\subsection{The main geometric idea} A commutative cartoon of our construction is perhaps instructive, it will be rendered less cartoonish in section 5. The main idea consists in birationally modifying the free commutative vertex algebra generated by $X$, namely the space $JX$ of maps $$\Delta:=\operatorname{spec}\mathbf{C}[[z]]\rightarrow X,$$ in a manner specified by the divisor $D$. By a birational modification we mean only that we define a variant of $JX$, functorial in the pair $(X,D)$, which agrees with $JX$ over the open piece $X\setminus D$.  Informally, we modify this by considering only maps $\varphi$ endowed with a certain decoration when $\varphi(0)$ lies in the boundary $D$. 

More precisely we will consider the space $J_{\operatorname{log}D}(X)$ formed from $X$ by universally adjoining a vector field tangent to $D$, much as $JX$ can be thought of as the result of universally adjoining a vector field to $X$, which on $JX$ corresponds to the infinitesimal translation on $\Delta$, as explained in 2.3.2. of \cite{Bei}. Basic properties of $J_{\operatorname{log}D}(X)$ are easily described in terms of the pair $(X,D)$, for example if $D$ is smooth then the associated variety in the sense of \cite{Ara} is described as  $$\operatorname{AssVar}(J_{\operatorname{log}D}X)=X\cup_{D} N_{D/X},$$ with $\mathbf{C}^{*}$ action contracting the normal fibre directions. There is a natural divisor corresponding to $D$, which is preserved by our vector field, so that the resulting structure can be thought of as a \emph{log commutative vertex algebra}. We will eventually see that taking log de Rham forms along this will produce a Lagrangian subalgebra of our main object of study $\Omega^{ch}_{X}(\operatorname{log}D)$.

\subsection{Motivation coming from $\chi_{y}$ genus} We will sketch now another motivating problem, before explaining precisely how we will fail to solve it. If $X$ is compact, taking the trace on $H^{*}(X,\Omega^{ch}_{X})$ of the operators $L_{0}$ and $J_{0}$, graded by variables $q$ and $y$,  produces the \emph{elliptic genus} of $X$, $\operatorname{Ell}_{X}(q,y)\in\mathbf{Z}[y,y^{-1}][[q]]$. It was shown in \cite{BorLib} that for $q=\exp(2\pi i\tau)$, $y=\exp(2\pi iz)$ and $X$ Calabi-Yau, $\operatorname{Ell}_{X}(\tau,z)$ is a \emph{Jacobi form}  with index $\frac{d}{2}$ and weight $0$. It is shown in \cite{Libg} that in general $\operatorname{Ell}_{X}(\tau,z)$ is only a \emph{quasi Jacobi form.} 

Note that here it is crucial that $X$ is compact. In order to see what might be done for non-compact $X^{*}$, a first approximation is suggested by looking at the conformal weight $0$ subspace, which is simply the de Rham complex of $X^{*}$. The resulting trace is then just the $\chi_{y}$ genus, which can be defined for non-compact $X^{*}$, thanks to Deligne's mixed Hodge theory, see \cite{Del}. The recipe is as follows: we fix a compactification $$j:X^{*}\rightarrow X$$ with boundary a simple normal crossings divisor $$D=\bigcup_{i\in I}D_{i}\subset X.$$ We have then the graded algebra of \emph{log de Rham} forms $$\Omega_{X}^{*}(\operatorname{log}D)\subset j_{*}\Omega^{*}_{U},$$ whose local sections are generated over $\Omega_{X}$ by $d\operatorname{log}f_{i}$ where the $f_{i}$ are rational functions cutting out irreducible components $D_{i}$ of $D$. We then take the  Euler characteristic of this. That the result is independent of the choice of compactification is a consequence of mixed Hodge theory. 

\begin{remark} We include a discussion of $\chi_{y}(U)$ for open varieties $U$, as it is perhaps unfamiliar to some readers. By Deligne's remarkable work the de Rham cohomology $H^{*}(U)$ is endowed with a mixed Hodge structure. Such amounts in particular to an increasing \emph{weight} filtration $$...\subset W^{i}H^{*}(U)\subset W^{i+1}H^{*}(U)\subset ...$$ such that the associated graded $\operatorname{Gr}_{W}^{i} H^{j}(U)$ has a pure Hodge structure of weight $i+j$. For example there is an isomorphism of weight $2$ pure Hodge structures $$\operatorname{Gr}_{W}^{1}(H^{1}(\mathbf{C}^{*}))\simeq H^{2}(\mathbf{P}^{1}).$$ Then we set $$\chi_y(U) = \sum_{k,p,q} (-1)^k h^{p,q}(\operatorname{Gr}_{p+q}^W H^k(U)) y^p.$$ Let us now take \emph{any} compactification $U\subset X$ where the complement $D=X\setminus U$ has simple normal crossings. Then it is a consequence of Deligne's construction of the weight filtration in terms of log de Rham complexes that we also have $$\chi_{y}(U)=\sum_{p}y^{p}\chi(X,\Omega^{p}_{X}(\operatorname{log}D)),$$ whence by Riemann-Roch there is a description as the integral \emph{on the compact space X} of the cohomology class $$\operatorname{ch}\Big(\bigwedge(y\Omega^{1}_{X}(\operatorname{log}D))\Big)\operatorname{Td}(X).$$  Note that the expression for $\chi_{y}$ in terms of the pair $(X,D)$ is far from obviously independent of the choice of pair, that it in fact is independent forms a (small) part of the magic of mixed Hodge theory.  \end{remark}

\begin{remark} Note that the definition $\chi_{y}(U)=\sum_{p}y^{p}\chi(X,\Omega^{p}_{X}(\operatorname{log}D))$ appears already as definition 2.10 of \cite{Baty}, where it is the specialization at $v=1$ of $E$, in the notation of \emph{loc. cit.} \end{remark}

\subsection{Results of this paper}The above suggests a natural guess for our chiral situation. We should define a logarithmic chiral de Rham complex associated to the pair $(X,D)$. This should be a sheaf of conformally graded vertex algebras on $X$, with weight $0$ subspace recovering $\Omega_{X}^{*}(\operatorname{log}D)$. Further, one expects $\mathcal{N}=2$ supersymmetry (resp. extended such) when $(X,D)$ is \emph{log Calabi-Yau} in the sense that there is a volume form on $X\setminus D$ with logarithmic poles along $D$. We show in section 3 that we have the following:

\begin{theorem} Let $(X,D)$ be a log pair with $X^{*}:=X\setminus D$. \begin{itemize}\item There is a sheaf  $\Omega^{ch}_{X}(\operatorname{log}D)$ of conformally graded vertex algebras with conformal weight $0$ subspace  $\Omega_{X}(\operatorname{log}D)$.\item A choice of volume form on $X^{*}$ endows $\Omega^{ch}_{X}(\operatorname{log}D)$ with a topological structure of rank $d=\operatorname{dim}(X)$. If $D$ is in fact anticanonical, then the choice of a top form on $X$ with log poles along $D$ endows $\Omega^{ch}_{X}(\operatorname{log}D)$ with an extended topological structure. \end{itemize}\end{theorem}

\begin{remark} We stress that this \emph{does not} solve the aforementioned problem of defining elliptic genera for open varieties! The resulting genera will in general be sensitive to the choice of a compactification, even if their limits at $\tau\rightarrow i\infty$ are not.  Nonetheless, we are confident that the above is the correct notion of chiral de Rham complex for pairs $(X,D)$, as the desiderata for such are quite substantial.\end{remark}

In section 4 we will study the cohomology of global sections, as a module for the topological vertex algebra.

\begin{example}\label{toric} Let us consider $X:=\mathbf{P}^{d}$ and $D:=\bigcup_{i}\big\{X_{i}=0\big\}$ the toric boundary. Then we produce from this a module for $\mathbf{top}\{d\}$ by taking cohomology of $\Omega^{ch}_{X}(\operatorname{log}D)$. We can compute the resulting bi-graded character to be $G(q,y)^{d}$ where we have set $$G(q,y)=\prod_{j=1}^{\infty}\frac{(1-q^{j-1}y)(1-q^{j}y^{-1})}{(1-q^{j})^{2}}.$$ A sketch proof appears later in the text, after the relevant definitions have been set up.\end{example}

We observe that $G(q,y)^{d}$ satisfies $G(q,qy)=(-y)^{-d}G(q,y)^{d}$, and is thus a section of the line bundle $\Theta^{\otimes d}$ on $E_{q}:=\mathbf{C}^{*}/q^{\mathbf{Z}}$, by definition the line bundle with automorphy factor $(-y)^{-d}$, see \ref{autom} below for a more precise discussion. This is a general fact when $D$ is anti-canonical, in which case we will say that $(X,D)$ is \emph{log CY}. In the following we set $$\mathcal{H}_{X,D}=H^{*}(X,\Omega^{ch}_{X}(\operatorname{log}D)).$$

\begin{theorem} Let $(X,D)$ be log CY of dimension $d$ and define $$\operatorname{Ell}_{X,D}(q,y)=\operatorname{tr}_{\mathcal{H}_{X,D}}(q^{L_{0}}(-y)^{J_{0}}).$$ Then we have $\operatorname{Ell}_{X,D}(q,y)\in\Gamma(E_{q},\Theta^{\otimes d})$.\end{theorem}

\begin{remark}In \cite{BorLib} an elliptic genus is defined for \emph{singular} varieties by taking resolutions and modifying the elliptic genus of the total space by certain $\vartheta$ quotients in the classes of exceptional divisors. In particular a notion of elliptic genus for pairs arises in \emph{loc. cit.} The construction of this note differs somewhat from that of \emph{loc. cit.}, in which it is unclear how to realise the genus as the trace of a sheaf of $\mathcal{N}=2$ vertex algebras. \end{remark}

In section 5, which we have included as we find it clarifying geometrically, we develop the geometric idea sketched in subsection \ref{geom}. In particular we identify a Lagrangian subalgebra of $\Omega^{ch}_{X}(\operatorname{log}D)$, which is a logarithmic analogue of the $c\gamma$ system on $X$, namely $\mathcal{V}^{c\gamma}_{X}:=\Omega(JX)$. Writing $\mathcal{V}^{c\gamma}_{X,D}$ for the logarithmic $c\gamma$ system, we prove the following, where relevant terms are defined properly in the sequel.

\begin{theorem} There is an isomorphism of sheaves of commutative vertex algebras $$\mathcal{V}^{c\gamma}_{X,D}\cong\Omega_{J_{\operatorname{log}D}(X)}(\operatorname{log}\widetilde{D}),$$ where $\widetilde{D}$ denotes the divisor cut out by the condition $\varphi(0)\in D$.\end{theorem}

\subsubsection{Some remarks on our geometric context}\label{context} We work throughout in the category of algebraic varieties over the field $\mathbf{C}$, although all of our results are valid in the broader context of algebraic varieties over a base field $k$ of characteristic $0$. We could also work in the holomorphic category and the results would be essentially unchanged. The chiral de Rham complex of \cite{MSV} makes sense also in the smooth category, but in this context the notion of logarithmic singularity is missing and so our results do not have clear analogues in the smooth context. In keeping with our algebro-geometric context we shall adopt standard algebro-geometric notation. For example by \emph{bundle} $\mathcal{E}$ the reader is to understand \emph{locally free sheaf} $\mathcal{E}$. Associating to $\mathcal{E}$ its total space $\operatorname{Tot}(\mathcal{E})\rightarrow X$ we obtain a \emph{geometric bundle}, and this association supplies an anti-equivalence between the respective categories of algebraic and geometric objects. It is perhaps also worth remarking on the presence of the word \emph{topological} in this text. This descriptor forms a part of the name of a specific vertex algebra, and then we say that a module over this vertex algebra is \emph{endowed with a topological structure}. It is not to say that at any point of this note our interest is in topological manifolds.

\subsection{Acknowledgements} We are very grateful to the two anonymous referees for their careful reading of this work. We also thank Fedya Malikov for providing feedback on a draft version.

\section{Chiral de Rham complex}\subsection{Vertex algebra language} We assume familiarity with the basic theory of vertex algebras, for which the reader can consult \cite{Kac} for an excellent introduction. Given a vertex algebra $V$, we let $\operatorname{vac}$ denote the vacuum element and $\partial$ denote the infinitesimal translation. If $v\in V$, we write $$v(z)=\sum v_{(i)}z^{-1-i}$$ for the corresponding field. We will sometimes abuse notation and identify a vector in a vertex algebra with the field it generates. If we have fixed a conformal grading of $V$ then for a vector $v$ of weight $\Delta(v)$ we define operators $v_{i}$ (note the lack of brackets around the subscript $i$) by $$v(z):=\sum v_{i}z^{-\Delta(v)-1-i}.$$ If $v$ and $w$ are vectors in $V$ then we abbreviate $v_{(-1)}w$ as $vw$, and caution the reader that this multiplication is neither commutative nor associative in general. \subsection{Construction of $\Omega^{ch}_{X}$} So as to fix some notation, we very briefly recall the definition of the chiral de Rham complex associated to a smooth variety $X$, complete definitions can be found in the original paper \cite{MSV}, which is our main reference for this subsection. 

 $\Omega^{ch}_{X}$ is a sheaf of vertex algebras on $X$. It is constructed first on an affine space $\mathbf{A}^{d}$ with coordinates $\gamma^{1},...,\gamma^{d}$. 

\begin{definition} $\Omega^{ch}_{\mathbf{A}^{d}}$ is defined to be the $bc\beta\gamma$ system, namely the vertex algebra generated by bosonic fields $\{\gamma^{i}(z),\beta_{i}(z)\}$ and fermionic fields $\{b_{i}(z), c^{i}(z)\}$, for $i=1,...,d$, subject only to the OPEs $$\gamma^{i}(z)\beta_{j}(w)\sim\frac{\delta_{ij}}{(z-w)}$$ $$b_{i}(z)c^{j}(w)\sim\frac{\delta_{ij}}{(z-w)}.$$\end{definition}\begin{remark} Informally the fields have the following meanings - $\gamma^{i}$ are coordinates on $\mathbf{A}^{d}$, with $\beta_{i}$ the corresponding dual derivations. $c^{i}$ and $b_{i}$ are then interpreted as one forms and odd vector fields respectively.\end{remark} \begin{remark}A notational caution - we have chosen the physicists' notation for our fields, in \cite{MSV} the fields are labelled $a,b,\varphi,\psi$. The translation from our notation to the notation of \emph{loc. cit.} is as follows; $\gamma\mapsto b$, $\beta\mapsto a$, $b\mapsto \psi$, $c\mapsto \varphi$.\end{remark}

\begin{definition} Completion along the point $0\in \mathbf{A}^{d}$ defines the vertex algebra $\Omega^{ch}_{\Delta^{d}}$ associated to a formal disc $\Delta^{d}$ with coordinates $\gamma^{1},...,\gamma^{d}$. \end{definition}

\begin{remark} Explicitly this means we allow power series in the variables $\gamma^{j}_{(-1)}$, which we interpret as our coordinates $\gamma^{j}$, see subsection 3.1 of \emph{loc. cit}.\end{remark}

In order to glue the above formal local models corresponding to discs $\Delta^{d}$, we must remove the dependence on a choice of formal coordinates at a point of $X$. This is achieved by the following theorem, which is one of the main results of \emph{loc. cit.} \begin{theorem}(\cite{MSV}) There is an action of the group $G_{d}$ of formal coordinate transformations on the above vertex algebra. The resulting sheaf of vertex algebras on $X$ is then denoted $\Omega^{ch}_{X}$. \end{theorem}

\subsection{Extended topological symmetry of $\Omega^{ch}_{X}$}

Relevant to our study is a certain vertex algebra, referred to as the \emph{topological vertex algebra at rank d}. This is a so-called \emph{topological twist} of the famous $\mathcal{N}=2$ superconformal vertex algebra.

\begin{definition} The \emph{topological vertex algebra}, $\mathbf{top}\{d\}$, of \emph{rank} $D$ is the super vertex algebra generated by vectors $\{L,J,Q,G\}$ such that $L$ and $J$ are bosonic and $Q$ and $G$ are fermionic. A full list of the non-trivial OPEs to which the corresponding fields are subject is found in \cite{MSV}, subsection 2.1. \end{definition}

\begin{example} We include some facts about $\mathbf{top}\{d\}$, so that the reader can get a feel for it. We have the following: \begin{itemize}\item $L$ generates a copy of the centreless Virasoro. This induces a conformal grading on  $\mathbf{top}\{d\}$. \item The even field $J(z)$ satisfies $$J(z)J(w)\sim\frac{d}{(z-w)^{2}},$$ so that $J$ generates a representation of the Heisenberg at level $d$. \item The odd fields satisfy the OPE $$Q(z)G(w)\sim\frac{L(w)}{(z-w)}+\frac{J(w)}{(z-w)^{2}}+\frac{d}{(z-w)^{3}}$$ and so in particular we have $[G_{0},Q_{0}]=L_{0}$. \end{itemize}\end{example}

\begin{definition} A vertex algebra $V$ equipped with a distinguished map $\mathbf{top}\{d\}\rightarrow V$ is referred to as a topological vertex algebra of rank $d$. \end{definition}

Then we have the following theorem, proven in \emph{loc. cit.} We will call a variety Calabi-Yau if it is equipped with a non-vanishing volume form.  \begin{theorem}(\cite{MSV}.) For $X$ Calabi-Yau, there are global sections $L,J,Q,G$ of  $\,\Omega^{ch}_{X}$ giving it the structure of a topological vertex algebra at rank $d$. Formally locally these are expressed explicitly as follows, where we sum over repeated indices: \begin{itemize}\item $ L=\partial \gamma^{i}\beta_{i}+\partial c^{i}b_{i}$
\item $J=c^{i}b_{i}$
\item $Q=\beta_{i}c^{i}$
\item $G=\partial \gamma^{i}b_{i}$.\end{itemize}\end{theorem}

\begin{remark} In fact, whilst the above vectors are only globally well defined in case $X$ is Calabi-Yau, the explicit form of their transformations implies that in all cases $J_{0}$, $L_{0}$ are defined. They are further semi-simple and so $\Omega^{ch}_{X}$ is bigraded by \emph{conformal weight} and \emph{fermion number}. Similarly $Q_{0}$ is defined in all cases. It is square zero and referred to as the \emph{chiral de Rham differential}. It restricts on the conformal weight $0$ subspace to the de Rham differential. \end{remark}

As first observed in \cite{Odake}, a natural extension of the topological vertex algebra at rank $d$ actually acts on $\Omega^{ch}_{X}$ for a Calabi-Yau $X$. This is essentially gotten by adding the volume form to the topological algebra inside $\Omega^{ch}_{X}$, see \emph{loc. cit.} for a definition. The resulting algebra will be referred to as the \emph{extended} topological vertex algebra at rank $d$, and denoted $\mathbf{top}\{d\}^{!}$. A vertex algebra equipped with a morphism from $\mathbf{top}\{d\}^{!}$ is said to possess the structure of an \emph{extended topological vertex algebra} at rank $d$.

\begin{theorem} (\cite{Odake}) The specification of a volume form on $X$ endows $\Omega^{ch}_{X}$ with the structure of an extended topological algebra at rank $d$. \end{theorem}

\section{The Logarithmic Chiral de Rham Complex} We now construct the main object of study of this note, namely the log chiral de Rham complex associated to a \emph{log pair} $(X,D)$ - a smooth variety $X$ and a simple normal crossings divisor $D$. We let $r$ denote the number of irreducible components $D_{i}$. We write $X^{*}:=X\setminus D$ throughout, and denote the inclusion of $X^{*}$ into $X$ as $j$. \subsection{Formal local models} Following the strategy of the construction of $\Omega^{ch}_{X}$, we first construct the corresponding object on a formal $d$-dimensional disc equipped with a simple normal crossings divisor. Let $\Delta^{d}$ be the formal $d$-disc with coordinates $\gamma^{i}$. Choose some $r\leq d$ and let $D_{r}$ denote the divisor $$\big\{\gamma^{1}...\,\gamma^{r}=0\big\}\subset\Delta^{d}.$$ Let $j$ denote the inclusion of the open part $$j:\Delta^{d}\setminus D_{r}\longrightarrow \Delta^{d}.$$

\begin{definition} We define the vertex algebra $$\Omega^{ch}_{\Delta^{d}}(\operatorname{log}D_{r})\subset j_{*}\Omega^{ch}_{\Delta^{d}\setminus D_{r}}$$ to be the subalgebra generated by the vectors $$\big\{\gamma^{j},\beta_{j},c^ {j},b_{j}\big\}_{j>r}\bigcup \big\{\gamma^{i},\frac{\partial \gamma^{i}}{\gamma^{i}},\gamma^{i}\beta_{i},\frac{c^{i}}{\gamma^{i}}, \gamma^{i}b^{i}\big\}_{i\leq r}.$$\end{definition}

 We will say that a section of $ j_{*}\Omega^{ch}_{\Delta^{d}\setminus D_{r}}$ is \emph{logarithmic relative to} $D_{r}$ if it lies in $\Omega^{ch}_{\Delta^{d}}(\operatorname{log}D_{r})$ and abbreviate this to just \emph{logarithmic} when context is clear. 

\begin{lemma}\label{locsym} For any $D_{r}\subset\Delta^{d}$, the sections $L,J,G,Q$ are all logarithmic.\end{lemma}

\begin{proof}
 We have $$Q=\sum_{i}\beta_{i}c^{i}$$ and so it suffices to show that each $\beta_{j}c^{j}$ is logarthmic. If $j>r$ this is obvious. For $j\leq r$ we can compute from Borcherds' formula that $$(\gamma^{j}\beta_{j})\big(\frac{c^{j}}{\gamma^{j}}\big)=\gamma^{j}\beta^{j}-\big(\frac{\partial \gamma^{j}}{\gamma^{j}}\big)\big(\frac{c^{j}}{\gamma^{j}}\big),$$ which we recognise as logarithmic. We have $$G=\sum_{i}\partial \gamma^{i}b_{i}$$ and similarly to above we need only deal with summands $\partial \gamma^{j}b_{j}$ with $j\leq r$. Then we may simply write this as $\big(\frac{\partial \gamma^{j}}{\gamma^{j}}\big)\big(\gamma^{j}b_{j}\big)$, where there are no extra terms in the normally ordered product as all relevant monomials generate mutually local fields. Now we note that $Q_{0}G=L$ and $Q_{1}G=J$, and logarithmic sections are by definition closed under the vertex operations, whence we are done.\end{proof}

\subsection{Globalisation}Recall we have fixed a log pair $(X,D)$ and we wish to construct a vertex subalgebra of $j_{*}\Omega^{ch}_{X^{*}}$.  In order to do this we must study the effect of coordinate transformations on our formal local model. Of course, we must restrict only to coordinate transformations which preserve the divisor $D_{r}\subset\Delta^{d}$, as otherwise there is no chance of invariance.

\begin{definition} We denote by $G_{d,r}$ the subgroup of formal coordinate transformations preserving $D_{r}$. \end{definition}

\begin{lemma}\label{coord} Formulae (3.17a)-(3.17d) of \cite{MSV} define an action of $G_{d,r}$ on the vertex algebra $\Omega^{ch}_{\Delta^{d}}(\operatorname{log}D_{r})$.\end{lemma}

\begin{proof} Fix a coordinate transformation $$\tilde{\gamma}^{i}=g^{i}(\gamma^{1},...,\gamma^{d});\,\,\gamma^{i}=f^{i}(\tilde{\gamma}^{1},...,\tilde{\gamma}^{d}).$$ We will use tildes to denote the images of the transformed fields, so that $\beta^{j}$ transforms to $\tilde{\beta}^{j}$ etc. 

We first turn to the claim of the lemma for fields $\frac{\partial\gamma}{\gamma},\frac{c}{\gamma}$ and $\gamma b$. According to \cite{MSV} the odd fields $c^{i}$ and $b^{i}$ transform exactly as differential forms and vector fields respectively, and the claim that $G_{r,d}$ maps $\gamma^{i}b_{i}$ and $\frac{c^{i}}{\gamma^{i}}$ into the subspace  $\Omega^{ch}_{\Delta^{d}}(\operatorname{log}D_{r})$ follows immediately from the fact that an automorphism preserving $D_{r}$ preserves log one forms (\emph{resp.} vector fields) along $D_{r}$. An identical argument holds for the fields $\frac{\partial\gamma^{i}}{\gamma^{i}}$.  

We are left to deal with the fields $\beta^{i}$, which have a complicated transformation property depending crucially on a correction term which is quadratic in the fermionic fields. Explicitly we have $$\tilde{\beta}^{i}=\beta^{j}\frac{\partial f^{j}}{\partial\tilde{\gamma}^{i}}(g^{1},...,g^{d})+\frac{\partial^{2}f^{k}}{\partial\tilde{\gamma}^{i}\partial\tilde{\gamma}^{l}}(g^{1},...,g^{d})\frac{\partial g^{l}}{\partial\gamma^{r}}c^{r}b^{k}.$$

 It would be something of a notational bother to check directly the claim for the fields $\gamma^{i}\beta_{i}\in \Omega^{ch}_{\Delta^{d}}(\operatorname{log}D_{r})$ so we instead argue as follows.

The expression $c^{i}b_{i}=(\frac{c^{i}}{\gamma^{i}})(\gamma^{i}b_{i})$ makes clear that the vectors $c^{i}b_{i}$ are transformed to logarithmic fields under the action of $G_{d,r}$ as both $\frac{c^{i}}{\gamma^{i}}$ and $\gamma^{i}b_{i}$ map to logarithmic elements by above - indeed $G_{d,r}$ acts by vertex automorphisms and logarthmic elements are closed under the vertex operations by definition. Now recall that the results of \cite{MSV} imply that the operator $Q_{0}$ is invariant under all coordinate transformations, even though the field $Q(z)$ is not. Recall also that lemma \ref{locsym} tells us that $Q$ is logarithmic and so $Q_{0}$ preserves logarithmic elements as they are closed under the vertex operations. Finally we conclude from the expression $$Q_{0}(\gamma^{i}b_{i})-c^{i}b_{i}=\gamma^{i}\beta_{i}.$$ Indeed we have shown that the left hand side maps to a logarithmic form by invariance of the operator $Q_{0}$ and the fact that both $\gamma^{i}b_{i}$ and $\gamma^{i}\beta_{i}$ do. It follows that the right hand side must as well.  \end{proof}

By a \emph{local section} of a sheaf $\mathcal{F}$ on $X$ we will mean a pair consisting of an open set $U$ and a section $s\in \Gamma(U,\mathcal{F})$. We will abuse notation, forgetting $U$ and simply write $s\in\mathcal{F}$ to denote this.

\begin{definition}We say that a local section $s\in j_{*}\Omega^{ch}_{X^{*}}$ is \emph{logarithmic with respect to $D$} at a point $p\in X$ if for some choice of coordinates $\{\gamma^{i}\}$ at $p$, with respect to which $D$ is given by $\{\gamma^{1}...\,\gamma^{r}=0\}$, the image of $s$ inside $j_{*}\Omega^{ch}_{\Delta^{d}\setminus D_{r}}$ lands inside the subalgebra $\Omega^{ch}_{\Delta^{d}}(\operatorname{log}D_{r})$. \end{definition}

\begin{remark} Note that this condition is vacuous if $p\in X^{*}$. \end{remark}

\begin{lemma} The condition that a local section at $p$ is logarthmic with respect to $D$ is independent of the choice of local coordinates at $p$. \end{lemma}

\begin{proof} This is the content of lemma \ref{locsym} above. \end{proof}

\label{maindef}\begin{definition} We say that a local section $s$ is \emph{logarithmic with respect to $D$} if it is logarthmic at all points $p\in X$. The subsheaf of $j_{*}\Omega^{ch}_{X^{*}}$ whose local sections are the logarithmic sections with respect to $D$ is denoted $\Omega^{ch}_{X}(\operatorname{log}D)$ and referred to as the log chiral de Rham complex. \end{definition}

\begin{lemma} The subsheaf $\Omega^{ch}_{X}(\operatorname{log}D)$ is closed under the vertex operations, and so forms a vertex subalgebra.\end{lemma} \begin{proof} This is immediate from construction as $\Omega^{ch}_{\Delta^{d}}(\operatorname{log}D_{r})\subset j_{*}\Omega^{ch}_{\Delta^{d}\setminus D_{r}}$ is defined to be a subalgebra, and the condition that a section is logarithmic is local. \end{proof}

\begin{remark} Notice that, unlike in the non-chiral case, we do not have $\Omega^{ch}_{X}\subset\Omega^{ch}_{X}(\operatorname{log}D)$ on account of the presence of tangent directions corresponding to the vectors $\beta_{i}$ and $b_{i}$. \end{remark}

\begin{lemma} The restriction of the conformal grading on $j_{*}\Omega^{ch}_{X^{*}}$ to $\Omega^{ch}_{X}(\operatorname{log}D)$ has weight $0$ subspace the algebra $\Omega_{X}(\operatorname{log}D)$. \end{lemma}

\begin{proof} This is immediate from the local construction. \end{proof}

\begin{lemma} The inclusion of the conformal weight $0$ subspace is a quasi-isomorphism: $$\big(\Omega_{X}(\operatorname{log}D),Q_{0}\big)\longrightarrow\big(\Omega^{ch}_{X}(\operatorname{log}D),Q_{0}\big).$$ In particular the hypercohomology of $(\Omega^{ch}_{X}(\operatorname{log}D),Q_{0})$ is isomorphic to the cohomology of the open piece $X^{*}$.\end{lemma}

\begin{proof} The first part is a standard fact about modules for the topological algebra, $G_{0}$ provides a homotopy contracting to the conformal weight $0$ subspace as we have $[G_{0},Q_{0}]=L_{0}$. 

 For the second we recall that the hypercohomology of the complex of log de Rham forms on $X$ computes the cohomology of $X^{*}$, cf. \cite{Del}.\end{proof}

\subsection{Topological symmetry of  $\Omega^{ch}_{X}(\operatorname{log}D)$} Recall from above that a choice of volume form on $X$ endows $\Omega^{ch}_{X}$ with the structure of an extended topological vertex algebra at rank $d$. 

\begin{remark} In keeping with our algebraic context, volume form is here meant to be understood as a global section $\operatorname{vol}$ of the sheaf $\Omega^{d}_{X}$ on the algebraic variety $X$ which induces an equivalence: $$\operatorname{vol}:\mathcal{O}\xrightarrow{\operatorname{\sim}}\Omega^{d}_{X}.$$\end{remark}
\begin{theorem}The choice of a volume form on $X^{*}$ endows $\Omega^{ch}_{X}(\operatorname{log}D)$ with a topological structure. If $D$ is further anticanonical then the choice of a volume form on $X^{*}$ with log poles along $D$ endows $\Omega^{ch}_{X}(\operatorname{log}D)$ with an extended topological structure.  \end{theorem}

\begin{proof} We first prove that there is topological symmetry as soon there exists a volume form on $X^{*}$. 

The main point is the following, recall that $X^{*}\subset X$  has complement a simple normal crossings divisor $D$. Suppose we have chosen a volume form $\omega$ on $X^{*}$, so that we have vectors $L, G, Q, J$ in $\Omega^{ch}_{X^{*}},$ generating a copy of the topological vertex algebra. Then each of these vectors \emph{automatically} has logarithmic singularities along $D$. It suffices to prove this for the vectors $G,Q$ as they generate the topological vertex algebra. Further, formulae of \cite{MSV} imply that $G$ is globally defined on all of $X$, and so it is only with $Q$ that we need concern ourselves. We give the proof in the case where $D$ has a single component for ease of notation. The general case follows from similar arguments. 

 Now the question of whether $Q$ is logarithmic is local and so we may assume that $X^{*}$ admits global coordinates $\phi:=(\phi_1,...,\phi_d)$ inducing the volume form in the sense that $$\omega:=d\phi_1...d\phi_d.$$ Further, again because the claim is purely local, we may assume that there are global coordinates $z=z_1,z_2,...,z_d$ on $X$ so that $D$ is given by $\{z=0\}$. Note that we \emph{may not} assume that the coordinates $\phi_{j}$ on $X^{*}$ are the restrictions of the $z_i$, indeed this would imply that $\omega$ extended to all of $X$. We can however assume that we have $$(\phi_1,\phi_2,...,\phi_d)=(z^{n},\frac{z_2}{z^{n_2}},...,\frac{z_d}{z^{n_d}}),$$ as we can write any function on $X^{*}$ in the form $\frac{g}{z^{n}}$ with $g$ a function on $X$.

 With respect to these coordinate systems we get sections $c_{z},b^{z},\gamma_{z},\beta^{z}$ of $\Omega^{ch}_{X}$ on all of $X$. We call these the \emph{$z$-vectors}. Similarly we get local sections, defined only over $X^{*}$, denoted $c_{\phi},b^{\phi},\gamma_{\phi},\beta^{\phi},$ which we call these the \emph{$\phi$-vectors}. The equality $$(\phi_1,\phi_2,...,\phi_d)=(z^{n},\frac{z_2}{z^{n_2}},...,\frac{z_d}{z^{n_d}})$$ explains how to obtain a rational expression for the $\phi$-vectors in terms of the $z$-ones. We now dispense with the $z$-superscripts. Henceforth we write $c=c_{z}, b=b^{z}, \gamma=\gamma_{z}$ and $\beta=\beta^{z}$ and write the $\phi$-vectors in terms of these. 

Then by construction we have $$Q=\sum c^{i}_{\phi}\beta^{\phi}_i$$ and we need to see that this is logarithmic \emph{in terms of the $z$-vectors}. Certainly it suffices to show that each of the summands $c^{i}_{\phi}\beta^{\phi}_{i}$ is logarithmic and so we do this. We focus on the first summand as the others are treated identically. We now dispense with the sub/superscript $1$ as we are working at a fixed index. We have $$c_{\phi}\beta^{\phi}=(\gamma^{n-1}c)(\gamma^{1-n}\beta)$$ and we claim that this logarithmic. Noting that $c$ and $\gamma$ are mutually local, we have from Borcherds' formula the equality: $$(c_{-1}\gamma^{n-1})_{-1}=\sum_{i}c_{-1-i}(\gamma^{n-1})_{-1+i}.$$ Upon acting on $\gamma^{1-n}\beta$, only terms with $i\geq 0$ will survive as any annihilation mode of $c$ will kill $\gamma^{1-n}\beta.$ Further, any term $i\geq 2$ will also act trivially on $\gamma^{1-n}\beta$ as it will contain at least two annihilation modes coming from $\gamma^{n-1}$, and there is only one $\beta$ to contract against. So only $i=0,1$ can occur and we find that $$(c_{-1}\gamma^{n-1})_{-1}(\gamma^{1-n}\beta)$$ $$=\Big(c_{-1}(\gamma^{n-1})_{-1}+c_{-2}(\gamma^{n-1})_{0}\Big)(\gamma^{1-n}\beta)$$ $$=\Big(c_{-1}(\gamma^{n-1})_{-1}+c_{-2}(n-1)\gamma^{n-2}_{-1}\gamma_{0}\Big)(\gamma^{1-n}\beta)$$ $$=c\beta+(n-1)\frac{\partial c}{\gamma}.$$ We must show that this is logarithmic. Now we have seen in \ref{locsym} that $c\beta$ is logarithmic and so it suffices to show that $\frac{\partial c}{\gamma}$ is logarithmic. We observe that $$\partial\Big(\frac{c}{\gamma}\Big)-\frac{c\partial{\gamma}}{\gamma^{2}}=\partial\Big(\frac{c}{\gamma}\Big)-\frac{c}{\gamma}\frac{\partial{\gamma}}{\gamma}=\frac{\partial c}{\gamma}.$$ Finally, $\frac{c}{\gamma}$ is logarithmic by definition and thus so too is $\partial\Big(\frac{c}{\gamma}\Big)$, as logarithmic elements are closed under all vertex operations. Further, $\frac{\partial\gamma}{\gamma}$ is also logarithmic by definition and then so too is $\frac{c}{\gamma}\frac{\partial\gamma}{\gamma}$ and the claim is proven.

Now we turn to the claim of \emph{extended topological symmetry}. In fact we will give only a brief account of this. Recall from \cite{Odake} that the extended topological vertex algebra is obtained by adding in the modes of the volume form. In this case we can take a top degree form on $X$ with logarithmic singularities along $D$, which by definition lives inside $\Omega^{ch}_{X}(\operatorname{log}D)$. \end{proof}

\begin{remark} As the account of \cite{Odake} is rather physical, the more mathematically inclined reader is invited to consult the author's short note \cite{Boua}, where details of the extended topological structure on $\Omega^{ch}_{X}$ are spelled out. In particular, the precise properties needed are outlined in 1.4 of \emph{loc. cit.} In this note we will not actually give a definition of the extended symmetry algebra acting on $\Omega^{ch}_{X}(\operatorname{log}D)$, and again we direct the curious reader to \cite{Boua}. Nonetheless, further in this text it will be seen that additional symmetry is acquired by the characters of the log Chiral de Rham in the case of a log CY pair $(X,D)$. This is, according to the main theorem of \cite{Boua}, a shadow of the richer structure afforded by the extended topological symmetry.  \end{remark}

\section{Character of cohomology} We assume in this section that $(X,D)$ is a log Calabi-Yau pair, and turn our attention to the space of (derived ) global sections, which as in the introduction we denote  $$\mathcal{H}_{X,D}=H^{*}(X,\Omega^{ch}_{X}(\operatorname{log}D)).$$ \begin{definition}We will call a module for $\mathbf{top}\{d\}$ \emph{finite type} if: \begin{itemize}\item $J_{0}$ and $L_{0}$ act semi-simply and for each pair of integers $(E,j)$ the subspace $V(E,j):=\big\{v\in V\,|\, L_{0}v=Ev,\,\, J_{0}v=jv\big\}$ is finite dimensional. \item $V(E,j)$ is non-zero only if $E\geq 0$ and for fixed $E$ there are only finitely many $j$ with $V(E,j)$ non-zero.\end{itemize}\end{definition}

\begin{definition} The character of a module $V$ for $\mathbf{top}\{d\}$ is defined as $$\operatorname{char}_{V}(q,y)=\operatorname{tr}_{V}\big(q^{L_{0}}(-y)^{J_{0}}\big),$$ provided the sum is well defined.\end{definition}

\begin{remark} We caution the reader that we have dispensed of the customary factor of $q^{-\frac{c}{24}}=q^{-\frac{d}{8}}$ as it will not concern us. \end{remark}

\begin{remark}If $V$ is finite type then we have $\operatorname{char}_{V}(q,y)\in\mathbf{Z}[y,y^{-1}][[q]]$. \end{remark}

\begin{definition} Let $\mathcal{E}$ be a vector bundle on $X$. Then we define the formal sum of bundles $$\operatorname{Ell}(\mathcal{E})(q,y)=\bigotimes_{n\geq 1}\wedge_{-yq^{n-1}}(\mathcal{E})\otimes\wedge_{-y^{-1}q^{n}}(\mathcal{E}^{*})\otimes\operatorname{Sym}_{q^{n}}(\mathcal{E})\otimes\operatorname{Sym}_{q^{n}}(\mathcal{E}^{*}).$$ \end{definition}

Inspecting the explicit form of the filtration of \emph{loc. cit.} we have the following simple lemma, which we state after recalling some algebra.

\begin{definition} If $V=\cup F^{i}V$ is a $\mathbf{Z}_{\geq 0}$- filtered sheaf, then the \emph{associated graded} sheaf is the object $$\operatorname{Gr}V:=\oplus_{i} F^{i}V/F^{i-1}V.$$ \end{definition}

\begin{lemma}\label{filt} The restriction to $\Omega^{ch}_{X}(\operatorname{log}D)$ of the filtration on $j_{*}\Omega^{ch}_{X^{*}}$ defined in \cite{MSV}, 3.27,  has the associated graded sheaf $\operatorname{Ell}(\Omega_{X}(\operatorname{log}D))(q,y)$ as a bigraded bundle.\end{lemma}

\begin{remark} We stress that the associated graded is \emph{quasi-coherent}, even though the sheaf we started with was not. This can be seen to be a consequence of the fact that the $0$-mode of a product of functions on $X$ is not the product of the $0$-modes, but is the product of the $0$-modes \emph{modulo annihilation terms}, which decrease filtration degree.\end{remark}

\begin{lemma} The module $\mathcal{H}_{X,D}$ is finite type if $X$ is proper.\end{lemma}\begin{proof} This now follows from lemma \ref{filt}, by finiteness of coherent cohomology on proper varieties.\end{proof}

\begin{definition}\label{autom} We denote by $\Theta$ the line bundle on $E_{q}:=\mathbf{C}^{*}/q^{\mathbf{Z}}$ which by definition has local sections on a $q^{\mathbf{Z}}$-invariant subset $U$ the holomorphic functions $\varphi:U\rightarrow\mathbf{C}$ such that $\varphi(qy)=(-y^{-1})\varphi(y)$. \end{definition}

\begin{remark} Note that the above definition is essentially a definition by descent along the $q^{\mathbf{Z}}$ torsor $\mathbf{C}^{*}\rightarrow E_{q}$. Indeed we have specified a $q^{\mathbf{Z}}$-equivariant structure on the trivial bundle on $\mathbf{C}^{*}$, which by construction gives us a bundle on $E_{q}$. In terms of the language of \emph{automorphy factors} this amounts to the choice of the unique cocycle $j\in H^{1}(q^{\mathbf{Z}},\mathcal{O}_{\operatorname{hol}}(\mathbf{C}^{*}))$ so that $j(q)=-y^{-1}$. \end{remark}

\begin{definition}We define $$\tilde{\vartheta}(q,y)=\prod_{j=1}^{\infty}(1-q^{j-1}y)(1-q^{j}y^{-1}),$$ $$\tilde{\vartheta}_{+}(q,y):=\prod_{j=1}^{\infty}(1-q^{j}y)(1-q^{j}y^{-1})=\frac{\tilde{\vartheta}(q,y)}{1-y},$$ $$G(q,y)=\frac{\tilde{\vartheta}(q,y)}{\tilde{\vartheta}_{+}(q,1)}.$$\end{definition}

The following lemma is very easily seen and so we do not supply a proof:

\begin{lemma}\label{simple} If $\epsilon_{i}$ are Chern roots of a bundle $\mathcal{E}$ then, writing $\operatorname{ch}$ for the Chern character, we have $$\operatorname{ch}(\operatorname{Ell}(\mathcal{E})(q,y))=\prod_{i=1}\frac{\tilde{\vartheta}(q,ye^{\epsilon_{i}})}{\tilde{\vartheta}_{+}(q,e^{\epsilon_{i}})}.$$\end{lemma}

\begin{theorem}\label{main} Let $(X,D)$ be a log Calabi-Yau pair of dimension $d$, then we have the following\begin{itemize}\item The character of global sections is elliptic, $\operatorname{Ell}_{X,D}(q,y)\in\Gamma\big(E_{q},\Theta^{\otimes d}\big).$\item The specialisation at $q=0$ is the $\chi_{y}$-genus of $X^{*}$. \item The $y=1$ specialisation is the Euler characteristic of $X^{*}$. \end{itemize}\end{theorem}

\begin{proof} First we notice that there is a morphism $$\lambda_{\operatorname{ell}}:K(X)\rightarrow K(X)((y))[[q]]$$ such that $\lambda_{\operatorname{ell}}(a+b)=\lambda_{\operatorname{ell}}(a)\lambda_{\operatorname{ell}}(b)$ and such that if $[\mathcal{E}]$ is the class of a vector bundle then we have $$\lambda_{\operatorname{ell}}([\mathcal{E}])=\big[\operatorname{Ell}(\mathcal{E})(q,y)\big]\in K(X)[y,y^{-1}][[q]]\subset K(X)((y))[[q]].$$ This follows by observing that the operation $\lambda_{\operatorname{ell}}$ takes direct sums to tensor products and then formally extending it to virtual vector bundles by stipulating $$\lambda_{\operatorname{ell}}(-E)=\lambda_{\operatorname{ell}}(E)^{-1}.$$ That the inverse is well defined is a consequence of the simple observation that $$[\operatorname{Ell}(E)(q,y)]=1+\operatorname{order}(y,q)$$ in $K((y))[[q]]$, and any such element is necessarily invertible. 

We introduce Chern roots $-\alpha_{i}$ of $T_{X}$ and denote by $\delta_{i}$ the first Chern class of the bundles $\mathcal{O}(-D_{j})$, where $D_{j}$ are the irreducible components of $D$. 

We recall that there is a short exact sequence $$0\rightarrow\Omega_{X}\rightarrow\Omega_{X}(\operatorname{log}D)\rightarrow\oplus_{j}\mathcal{O}_{D_{j}}\rightarrow 0,$$ and we apply to this the morphism $\lambda_{\operatorname{ell}}$.

We claim now that the Chern character of $\operatorname{Ell}(\Omega_{X}(\operatorname{log}D))(q,y)$ is equal to $$G(q,y)^{r}\prod_{i=1}^{d}\frac{\tilde{\vartheta}(q,ye^{\alpha_{i}})}{\tilde{\vartheta}_{+}(q,e^{\alpha_{i}})}\prod_{j=1}^{r}\frac{\tilde{\vartheta}_{+}(q,e^{\delta_{j}})}{\tilde{\vartheta}(q,ye^{\delta_{j}})}.$$  Let us see this explicitly: first note we have $$[\Omega_{X}(\operatorname{log}D)]=[\Omega_{X}]+\sum_{i}[\mathcal{O}_{D_{i}}] \in K(X)$$ and so we can hit this with the morphism $\lambda_{\operatorname{ell}}$ to obtain the identity in $K(X)((y))[[q]]$:

$$\lambda_{\operatorname{ell}}([\Omega_{X}(\operatorname{log}D)])=\lambda_{\operatorname{ell}}([\Omega_{X}])\prod_{i}\lambda_{\operatorname{ell}}([\mathcal{O}_{D_{i}}]).$$ We identify now the Chern characters of the individual terms in the product, noting that the Chern character is a morphism of algebras. Firstly, by \ref{simple} we have $$\operatorname{ch}(\lambda_{\operatorname{ell}})([\Omega_{X}])=\prod_{i=1}^{d}\frac{\tilde{\vartheta}(q,ye^{\alpha_{i}})}{\tilde{\vartheta}_{+}(q,e^{\alpha_{i}})}.$$ Now we identify the terms coming from the factors with $\mathcal{O}_{D_{i}}$. We have $$[O_{D_{i}}]=1-[\mathcal{O}_{X}(-D_{i})]\in K(X),$$ where $\mathcal{O}_{X}(-D_{i})$ is the line bundle of functions vanishing along $D_{i}$. Since $\lambda_{\operatorname{ell}}$ takes sums to products we have $$\lambda_{\operatorname{ell}}([\mathcal{O}_{D_{i}}])=\lambda_{\operatorname{ell}}(1)\lambda_{\operatorname{ell}}([\mathcal{O}_{X}(-D_{i})])^{-1}.$$ Now it is clear that we have $\operatorname{ch}\lambda_{\operatorname{ell}}(1)=G(q,y)$, indeed this is a special case of \ref{simple} with $\mathcal{E}=\mathcal{O}$. It remains only to note that, again by \ref{simple}, we have $$\operatorname{ch}\lambda_{\operatorname{ell}}([\mathcal{O}(-D_{i})])=\frac{\tilde{\vartheta}(q,ye^{\delta_{i}})}{\tilde{\vartheta}_{+}(q,e^{\delta_{i}})}.$$ Putting together the various terms in the product we have the desired expression for $\operatorname{ch}\lambda_{\operatorname{ell}}([\Omega_{X}(\operatorname{log}D)]).$

We now prove the claim about ellipticity of characters. Crucially we have the following equality: $$\sum_{i}\alpha_{i}=\sum_{j}\delta_{j}\in H^{2}(X).$$ This  follows from the fact that $D$ is anti-canonical, which implies that we have an equivalence of line bundles $$\mathcal{O}(-D)\simeq\bigotimes_{i}\mathcal{O}(-D_{i})\simeq \Omega_{X}^{d}.$$ Taking first Chern classes produces the desired equality. Recalling now the well-known identity $\tilde{\vartheta}(q,qy)=(-y^{-1})\tilde{\vartheta}(q,y)$ we note that upon substitution, $y\mapsto qy$, the term of $G(q,y)^{r}$  picks up a factor $(-y)^{-r}$, the term corresponding to $\alpha_{i}$ picks up a factor of $(-ye^{\alpha_{i}})^{-1}$ and the term corresponding to $\delta_{j}$ picks up a factor of $(-ye^{\delta_{j}})$. Multiplying this all together we find exactly $$(-y)^{-r}\prod_{i=1}^{d}(-ye^{\alpha_{i}})^{-1}\prod_{j=1}^{r}(-ye^{\delta_{j}})=(-y)^{-d}e^{(\sum_{j}\delta_{j}-\sum_{i}\alpha_{i})}=(-y)^{-d}$$ as claimed. 

That the $q=0$ specialisation is as claimed is the content of lemma 3.5.

That the $y=1$ specialisation is as claimed follows from lemma 3.6 and the fact that the $y=1$ specialisation is computed as an Euler characteristic, whence we can safely turn on the differential $Q_{0}$.

\end{proof}

\begin{example} We return to \ref{toric}, and sketch the proof of the formula stated above. Recall that we are computing $\operatorname{Ell}_{(\mathbf{P}^{d},D)}(q,y)$ where $D=\cup_{i} D_{i}$ the (anticanonical) toric boundary divisor. We let $X_{i}$ denote homogeneous coordinates on $\mathbf{P}^{d}$ and we note that $\frac{dX_{i}}{X_{i}}$ are global sections of $\Omega_{\mathbf{P}^{d}}(\operatorname{log}D)$. To see this it suffices to note that they are $\mathbf{C}^{*}$ equivariant sections of the corresponding complex of log forms on $\mathbf{C}^{d+1}\setminus \{0\}$. The Euler sequence implies that their sum is zero and thus an exact sequence $$0\rightarrow\mathcal{O}\rightarrow\mathcal{O}^{d+1}\rightarrow\Omega_{\mathbf{P}^{d}}(\operatorname{log}D)\rightarrow 0.$$ We deduce then that we have $$[\Omega_{\mathbf{P}^{d}}(\operatorname{log}D)]=d\in K(\mathbf{P}^{d})((y))[[q]].$$  We now hit this with $\lambda_{\operatorname{ell}}$ and use $\lambda_{\operatorname{ell}}(1)=G(q,y)$. Now the $K$-theoretic integral $$\chi:K(\mathbf{P}^{d})\rightarrow\mathbf{Z}$$ satisfies $\chi(1)=1$ and we are done.\end{example}

\section{Spaces of log Jets} \subsection{Jets and commutative vertex algebras} In this section we comment on the geometry behind our constructions. The results presented are simple and some proofs are only sketched, but we feel that the underlying geometry is clarifying.

Recall, for example from \cite{KapVass} section 2, that to a scheme $Y$ there is associated an infinite dimensional scheme parameterising maps $\Delta\rightarrow Y$, referred to as the \emph{jet space} of $Y$ and denoted $JY$. Further, the infinitesimal action of the vector field $\partial$ on $\Delta$ induces a global vector field on $Y$, which we also denote $\partial$. If $Y=\operatorname{spec}A$ is affine then so is $JY$, with algebra of functions denoted $JA$. We can present $JA$ as the algebra generated by symbols $a(n)$, for $a\in A$ and $n\geq 0$, subject only to the relations $$(ab)(n)=\sum_{i+j=n} a(i)b(j).$$ $\partial$ is then described by $$\partial a(n)=(n+1)a(n+1).$$ Recall that a \emph{commutative} vertex algebra is a vertex algebra $V$ such that $v_{(i)}w=0$ for all $v,w\in V$ and $i\geq 0$. Then it is a theorem of Borcherds that a commutative vertex algebra is equivalent to the data of pair consisting of a commutative algebra and a derivation. As such $JA$ is naturally a commutative vertex algebra and it is clear from the above presentation that is freely generated by $A$ as such, see 2.3.2. of \cite{Bei} for a thorough discussion. \begin{remark}We will refer to $JY$ as a commutative vertex scheme, with the meaning self-evident. It is the commutative vertex scheme freely generated by $Y$.\end{remark}

\begin{example} If $Y=\mathbf{A}^{1}$ with coordinate $x$ then $JY$ is isomorphic to an infinite dimensional affine scheme with coordinates $x_{i}$, interpreted as $i$-jets of a map $\Delta\rightarrow Y$. As such, $\mathcal{O}(JY)$ is the free commutative vertex algebra generated by an element $x$, and we have $x_{i}=\frac{\partial^{i}x}{i!}$. \end{example}

\subsection{Log commutative vertex algebras} Our goal in this subsection is to give a geometric interpretation of a certain \emph{Lagrangian} subalgebra of $\Omega^{ch}_{X}(\operatorname{log}D)$. We use the term \emph{Lagrangian} to refer to a maximal commutative subalgebra. The analogue in the non-logarithmic situation is the subalgebra generated locally generated by the fields $\gamma^{i}$ and $c^{i}$, which globalises to the commutative vertex algebra of differential forms on $JX$, according to the transformation formulae of \cite{MSV}. We will denote the resulting sheaf of commutative vertex algebras $\mathcal{V}^{c\gamma}_{X}:=\Omega(JX)$.

\begin{definition} We write $\mathcal{V}^{c\gamma}_{X,D}$ for the commutative subalgebra of $\Omega^{ch}_{X}(\operatorname{log}D)$ formed by intersecting $\Omega^{ch}_{X}(\operatorname{log}D)$ with $j_{*}\mathcal{V}_{X^{*}}$, and call it the log $c\gamma$ system.  \end{definition}

Inside $\mathcal{V}^{c\gamma}_{X,D}$ we have a sheaf of subalgebras defined using only $\gamma$ fields.\begin{definition} The resulting scheme over $X$ is denoted $J_{\operatorname{log}D}(X)$, and referred to as the space of \emph{log jets}.\end{definition}

 Inside this there is a natural divisor, $\widetilde{D}$, which is just the pull-back of $D$. 

\label{divinv}\begin{lemma} The natural global vector field on $J_{\operatorname{log}D}(X)$ preserves $\widetilde{D}$.\end{lemma}

\begin{proof} This is a local condition, and locally amounts to the fact that $$\partial\gamma^{i}=\gamma^{i}\big(\frac{\partial\gamma^{i}}{\gamma^{i}}\big),$$  where the rightmost term is a regular function on $J_{\operatorname{log}D}(X)$ by definition. \end{proof}

Recall that we call a derivation, $\partial$, \emph{logarithmic} (with respect to $D$) if $\partial$ preserves the ideal sheaf defining $D$. 

Our goal is to describe this commutative vertex algebra in a manner analogous to the aforementioned description of $\mathcal{V}^{c\gamma}_{X}$. We will see that it can be described as the space of \emph{log jets} mentioned in the introduction, the definition of which we recall below.

\begin{definition} A log commutative vertex scheme (\emph{lcvs}) is a tuple $(X,D,\partial)$, with $(X,D)$ a log pair and $\partial$ a vector field on $X$ tangent to the divisor $D$. \end{definition}

By lemma \ref{divinv} we have a lcvs $(J_{\operatorname{log}D}(X),\widetilde{D},\partial)$.

\begin{lemma} $(J_{\operatorname{log}D}(X),\widetilde{D},\partial)$ is the free lcvs generated by $(X,D)$, that is to say the right adjoint to the forgetful functor $(Y,D_{Y},\partial_{Y})\mapsto (Y,D_{Y})$ is given by log jets. \end{lemma}

\begin{proof} Given a lcvs $(Y,D_{Y},\partial_{Y})$ and a map $f:(Y,D_{Y})\rightarrow(X,D)$ we must show that it extends uniquely to a morphism of lcvs $$(Y,D_{Y},\partial_{Y})\rightarrow(J_{\operatorname{log}D}(X),\widetilde{D},\partial).$$ Locally we extend this map uniquely by compatibility with derivations, noting that the image of $\frac{\partial\gamma^{i}}{\gamma^{i}}$ is well defined as $f$ preserves the log structures and $\partial_{Y}$ is tangent to $D_{Y}$. \end{proof}

\begin{theorem} There is isomorphism of commutative vertex algebras $$\mathcal{V}^{c\gamma}_{X,D}\cong\Omega_{\operatorname{log}\widetilde{D}}(J_{\operatorname{log}D}(X)).$$\end{theorem}

\begin{proof} There is essentially nothing more to be done, we note only that we need to allow log forms along $\widetilde{D}$ to account for the local generating fields $\frac{c^{i}}{\gamma^{i}}\in\mathcal{V}^{c\gamma}_{X,D}$.\end{proof}

\subsection{Miscellaneous remarks regarding the geometry log jet space}

We include as a remark some additional facts about the space of log jets into pairs, with proofs omitted.

\begin{remark}\begin{itemize}We let $f_{i}$ be some rational functions locally cutting out the components $D_{i}$ and we let $\pi$ denote the natural map from the space of log jets to the space of jets. \item Points of the space $J_{\operatorname{log}D}(X)$ are identified with data $$\{\varphi:\Delta\rightarrow X,\,\, \{\psi_{i}\}_{i=1}^{r}\in\Omega^{1}_{\Delta}\,\operatorname{with}\,d\varphi^{*}(f_{i})=\psi_{i}\varphi^{*}(f_{i})\}.$$
\item The map $\pi$ is an isomorphism on $X^{*}$.
\item If $\varphi(0)\in D$, then $\varphi$ lifts under $\pi$ only if $\varphi$ factors through $D$.
\item A lift of a jet factoring through the depth $l$ stratum of the stratification induced by $D$ is equivalent to a relative jet into the normal bundle to this stratum.
\item Given a curve $\Sigma$, there are multipoint versions of $J_{\operatorname{log}D}(X)$ living over the Ran space $\operatorname{Ran}(\Sigma)$, as in \cite{KapVass} in the non logarithmic case. \end{itemize}\end{remark}

\section{Declarations}\subsection{Funding, Associated Data and Conflicts of Interest} The author received no funding for this work. The author has no conflicts of interest to report. This manuscript has no associated data. This is the accepted version of the same named paper appearing in \emph{Annales Henri Poincar\'e,} with DOI https://doi.org/10.1007/s00023-025-01627-2.


\begin{thebibliography}{9}

\bibitem{Ara}


T. Arakawa, A. Moreau,
\textit{Sheets and associated varieties of affine vertex algebras,}
Adv. Math, 320 (2017), 157-209.

\bibitem{Baty}
V. Batyrev,
\textit{Non-Archimedean integrals and stringy Euler numbers of log-terminal pairs,} 
J. Eur. Math. Soc. (JEMS) 1 (1999), no. 1, 5-33.


\bibitem{Bei}
A. Beilinson, V. Drinfeld,
\textit{Chiral algebras,}
AMS Colloquium publications, vol. 51.

\bibitem{BorLib}
L. Borisov, A. Libgober,
\textit{Elliptic genera of toric varieties and applications to mirror symmetry,}
Invent. Math. 140 pp. 453-485. 

\bibitem{BorLib2}
L. Borisov, A. Libgober,
\textit{Elliptic genera of singular varieties,}
Duke Math. J. 116(2): 319-351

\bibitem{Boua}
E. Bouaziz,
\textit{Spectral flow equivariance for Calabi-Yau sigma models,}
	arXiv:2407.02152

\bibitem{Del}
P. Deligne,
\textit{Theorie de Hodge II,}
Pub. Math. de l'IHES, vol 40, pp. 5-57.

\bibitem{Kac}
V. Kac,
\textit{Vertex Algebras for Beginners,}
University Lecture Series, AMS (1998)

\bibitem{KapVass}
M. Kapranov, E. Vasserot,
\textit{Vertex algebras and the formal loop space,}
Pub. Math. de l'IHES, vol. 100 (2004), pp. 209-269.

\bibitem{Kap}
A. Kapustin,
\textit{Chiral de Rham complex and half twisted sigma model,} 
arXiv:hep-th/0504074

\bibitem{Libg}
A. Libgober,
\textit{Elliptic genera, real algebraic varieties and quasi-Jacobi forms.}
Topology of stratified spaces, 95-120, Math. Sci. Res. Inst. Publ., 58, Cambridge Univ. Press, Cambridge, 2011.

\bibitem{Odake}
S. Odake,
\textit{Extension of N=2 superconformal algebra and Calabi-Yau compactification,}
Mod. Phys. Lett. A4 (1989), 347-381.

\bibitem{MSV}
F. Malikov, V. Schechtman, A. Vaintrob,
\textit{Chiral de Rham Complex,}
Commun. Math. Phys. 204, 439-473

\bibitem{Witt}
E. Witten,
\textit{Mirror manifolds and topological field theory,}
Essays on Mirror Manifolds, pp. 120-158.


















\end{thebibliography}
\end{document}